\documentclass[11pt,reqno]{amsart}

\usepackage[text={160mm,240mm},centering]{geometry}            % See geometry.pdf to learn the layout options. There are lots.
\usepackage{tikz}
\usepackage{tikz-3dplot}
\usetikzlibrary{arrows.meta}%画箭头，并非必需
\usetikzlibrary{commutative-diagrams}%交换图
\usetikzlibrary{cd}
\usepackage{amssymb}
\usepackage{amsmath}
\usepackage{longtable}
\usepackage{textcomp}
\usepackage{bm}
\newtheorem{theorem}{Theorem}[section]
\newtheorem{lemma}[theorem]{Lemma}
\newtheorem{proposition}[theorem]{Proposition}

\newtheorem{conjecture}[theorem]{Conjecture}
\newtheorem{corollary}[theorem]{Corollary}
\newtheorem{remark}[theorem]{Remark}

\theoremstyle{definition}
\newtheorem{definition}[theorem]{Definition}

\numberwithin{equation}{section}

\begin{document}

\title[hypersurface singularities]{The Quotient of Milnor Number by Tjurina Number of Hypersurface Singularities in Arbitrary Characteristic}

%    Information for first author
\author{Hongrui Ma}
%    Address of record for the research reported here
\address{Department of Mathematical Sciences,
Tsinghua University, Beijing, 100084, P. R. China.}
%    Current address
%\curraddr{Department of Mathematics and Statistics,
%Case Western Reserve University, Cleveland, Ohio 43403}
\email{mhr23@mails.tsinghua.edu.cn}
%    \thanks will become a 1st page footnote.
%\thanks{The first author was supported in part by NSF Grant \#000000.}

\author{Huaiqing Zuo}
\address{Department of Mathematical Sciences,
Tsinghua University,
Beijing, 100084, P. R. China.}
\email{hqzuo@mail.tsinghua.edu.cn}
\thanks{Zuo is supported by NSFC Grant 12271280 and BJNSF Grant 1252009}
%\thanks{Zuo is supported by NSFC Grant 12271280 and BJNSF Grant 1252009}

%    General info
%\subjclass[2020]{Primary 54C40, 14E20; Secondary 46E25, 20C20}

%\date{January 1, 2001 and, in revised form, June 22, 2001.}

%\dedicatory{This paper is dedicated to our advisors.}

%\keywords{algebraic geometry, singularity theory}

\begin{abstract}
In this paper, we use Hilbert-Samuel multiplicity, Hilbert-Kunz multiplicity, and s-multiplicity to establish a sharp upper bound for the quotient of the generalized Milnor number and the Tjurina number for isolated hypersurface singularities of any dimension in positive characteristic. Using this result, we also derive an upper bound for the quotient of the Milnor number $\mu$ and the Tjurina number $\tau$ for isolated hypersurface singularities of any dimension in characteristic zero. In particular, as a corollary, we obtain that for an isolated surface singularity $(f,0) \subset (\mathbb{C}^3,0)$, $\frac{\mu(f)}{\tau(f)} < \frac{3}{2}$, which answers a conjecture of P. Almir\'{o}n. This is also a weak version of Durfee's conjecture. We have also constructed a family of hypersurface singularities of any dimension for which $\frac{\mu}{\tau}$ tends to the bound we get, which means that the bound is sharp, and at the same time answers an open problem raised by P. Almir\'{o}n.

	
\end{abstract}

\maketitle

\section{introduction}
By a hypersurface singularity at the origin over $K$ we mean a power series $f \in \mathfrak{m}^2 \subset R=K[[x_1,...,x_n]]$, where $K$ is a field of arbitrary characteristic and $\mathfrak{m}$ is the maximal ideal of $R$. Two fundamental invariants are the {\em Milnor number} and the {\em Tjurina number}, which are given by the {\em Jacobian ideal} $j(f)=(\frac{\partial f}{\partial x_1},\dots,\frac{\partial f}{\partial x_n})$ and {\em Tjurina ideal} $tj(f)=(f,\frac{\partial f}{\partial x_1},\dots,\frac{\partial f}{\partial x_n})$. The Milnor number $\mu(f)$ and Tjurina number $\tau(f)$ are defined as
\begin{equation}\label{def-mu-tau}
    \mu(f)=\dim_K R/j(f) \quad \mathrm{and} \quad \tau(f)=\dim_K R/tj(f).
\end{equation}
The hypersurface singularity $f$ is isolated if and only if $\tau(f)< \infty$, while in characteristic zero, $\tau(f)<\infty$ is equivalent to $\mu(f)<\infty$.

For a long time, people have been paying attention to the relationship between the Milnor number and the Tjurina number of an isolated hypersurface singularity. The inequality $\frac{\mu}{\tau} \geq 1$ follows easily from the definition. In 1971, K. Saito \cite{saito-mu-tau} showed that for an isolated hypersurface singularity $f$ over $\mathbb{C}$, $\frac{\mu}{\tau}=1$ if and only if $f$ is analytically equivalent to a weighted homogeneous polynomial. For the upper bound of $\frac{\mu}{\tau}$, Y. Liu \cite{liuyongqiang} first gave a result
\begin{equation}\label{bound-liuyongqiang}
    \frac{\mu}{\tau} \leq n
\end{equation}
over $\mathbb{C}$. This work is related to the work of J. Briançon, H. Skoda \cite{skoda}, in which they defined the Briançon-Skoda exponent as $$e^{BS}(f)=\min\{e \in \mathbb{N} \mid f^e \in j(f)\} $$ and showed that $e^{BS}(f) \leq n$ over $\mathbb{C}$. Using the same method as \cite{liuyongqiang}, one can show that in the complex field, $$1 \leq \frac{\mu}{\tau} \leq e^{BS}(f).$$ However, this bound is far from sharp.

The sharper bound was first conjectured by A. Dimca and G.-M. Greuel \cite{greuel-conj} in 2018:
\begin{conjecture}\label{conj-4/3}
    $\frac{\mu}{\tau}<\frac{4}{3}$ for any isolated hypersurface singularity $f \in \mathbb{C}[[x,y]]$.
\end{conjecture}
They also constructed a sequence of plane curve singularities that asymptotically achieve the $\frac{4}{3}$ bound. Subsequently, this conjecture was proved in the case of irreducible plane curve and semi‑quasi‑homogeneous singularities. In 2022, P. Almir\'{o}n \cite{quotient-4/3} proved Conjecture \ref{conj-4/3} completely using formulas for the geometric genus. He also posed a new conjecture of dimension $2$:
\begin{conjecture}\label{conj-3/2}
    $\frac{\mu}{\tau}<\frac{3}{2}$ for any isolated hypersurface singularity $f \in \mathbb{C}[[x,y,z]]$.
\end{conjecture}
Interestingly, Conjecture \ref{conj-3/2} is a weak version of the following Conjecture \ref{conj-durfee} (still open), which is known as Durfee's conjecture (see Proposition \ref{durfee-3/2}).
\begin{conjecture}[\cite{conj-durfee}]\label{conj-durfee}
    For an isolated hypersurface singularity $f \in \mathbb{C}[[x,y,z]]$, $$6p_g \leq \mu,$$ with equality only when $\mu=0$, where $p_g$ is the geometric genus of the singularity.
\end{conjecture}
The study of Conjecture \ref{conj-3/2} is highly dependent on Conjecture \ref{conj-durfee}. As a weak version, since Conjecture \ref{conj-durfee} holds for singularities of one of the following types: (i) multiplicity 3, (ii) absolutely isolated singularity, (iii) suspension of the type $\{f(x,y)+z^n\}$, (iv) the link of the singularity is an integral homology sphere, (v) the topological Euler characteristic of the exceptional divisor of the minimal resolution is positive (see \cite{durfee-i,durfee-ii,durfee-iii,durfee-iv,durfee-v,durfee-vi}), Conjecture \ref{conj-3/2} holds under the same conditions.

Although there have been many works on the upper bound for $\frac{\mu}{\tau}$, these works are all based on fields of characteristic zero. The main reason is that for isolated hypersurface singularities in positive characteristic, it is possible that $\mu=\infty$. In 2019, A. Hefez, J. H. O. Rodrigues, and R. Salom\~ao \cite{mu-O} generalized the Milnor number as
\begin{equation}
    \begin{aligned}
        \mu(\mathcal{O}_f)&=e(tj(f)),\text{ the Hilbert-Samuel multiplicity of } tj(f),\\
        &=\min \{\mu(uf)\mid u \in R^{\times}\},\text{ if }K \text{ is an infinite field},
    \end{aligned}
\end{equation}
where $R^\times$ means the units of $R$. $\mu(\mathcal{O}_f)$ is always finite and coincides with $\mu(f)$ in characteristic zero. The generalized Milnor number allows us to discuss the quotient of Milnor number and Tjurina number in positive characteristic.

In this paper, we get a better bound on the quotient of the Milnor number and the Tjurina number for any dimension in positive characteristic with the help of Hilbert-Samuel multiplicity, Hilbert-Kunz multiplicity and $s$-multiplicity. Then we deduce the result to fields of characteristic zero. In conclusion, we get

\begin{theorem}\label{introduction-main-thm}
    For $n \geq 2$, let $f \in R=K[[x_1,\dots,x_n]]$ define an isolated hypersurface singularity at the origin. If $\mathrm{char}K>0$, then
    \begin{equation}\label{introduction-main-bound}
        \frac{\mu(\mathcal{O}_f)}{\tau(f)} \leq \frac{1}{\mathcal{H}_{\frac{n+1}{2}}(n)-\mathcal{H}_{\frac{n-1}{2}}(n)},
    \end{equation}
    if $\mathrm{char}K=0$, then
    \begin{equation}\label{introduction-main-bound-C}
        \frac{\mu(f)}{\tau(f)} < \frac{1}{\mathcal{H}_{\frac{n+1}{2}}(n)-\mathcal{H}_{\frac{n-1}{2}}(n)},
    \end{equation}
    where $\mathcal{H}_s(n)=\sum_{i=0}^{\lfloor s \rfloor}\frac{(-1)^i}{n!}\binom{n}{i}(s-i)^n$ for $0<s \leq n$.
\end{theorem}

We can see that for every $n \geq 2$, the bound in (\ref{introduction-main-bound-C}) is sharper than the bound in (\ref{bound-liuyongqiang}) given by Liu. For $n=2,3$, the bound given in (\ref{introduction-main-bound-C}) equals $\frac{4}{3},\frac{3}{2}$, which coincides with Conjecture \ref{conj-4/3}, \ref{conj-3/2}. Therefore, we prove Conjecture \ref{conj-3/2} independently, without relying on Durfee's conjecture.

A. Dimca and G.M. Greuel \cite[Example 4.1]{greuel-conj} showed that $\frac{4}{3}$ is sharp for $n=2$. Wahl \cite[Example 4.7]{Wahl} showed that $\frac{3}{2}$ is sharp for $n=3$. In this paper, we have constructed a family of isolated hypersurface singularities with $\frac{\mu}{\tau}$ tending to $\frac{1}{\mathcal{H}_{\frac{n+1}{2}}(n)-\mathcal{H}_{\frac{n-1}{2}}(n)}$ for all $n$. Therefore, the bound in Theorem \ref{introduction-main-thm} is sharp for all $n\geq 2$.\iffalse As a corollary, we answer part of the question raised by P. Almir\'{o}n \cite[Problem 1]{quotient-4/3}:

\begin{corollary}
    Let $(X,0)\subset (\mathbb{C}^N,0)$ be an isolated complete intersection singularity of codimension $1$. Then $\frac{1}{\mathcal{H}_{\frac{d+1}{2}}(d)-\mathcal{H}_{\frac{d-1}{2}}(d)}$ is the optional rational number such that $$\frac{\mu}{\tau} \leq \frac{1}{\mathcal{H}_{\frac{d+1}{2}}(d)-\mathcal{H}_{\frac{d-1}{2}}(d)},$$ where optimal means that there exists a family of singularities such that $\frac{\mu}{\tau}$ tends to $\frac{1}{\mathcal{H}_{\frac{d+1}{2}}(d)-\mathcal{H}_{\frac{d-1}{2}}(d)}$ when the multiplicity at the origin tends to infinity.
\end{corollary}
\fi

\section{Hilbert-Samuel multiplicity, Hilbert-Kunz multiplicity, and $s$-multiplicity}
We begin with some properties of Hilbert-Samuel multiplicity, Hilbert-Kunz multiplicity, and $s$-multiplicity. We will use these results throughout the paper without explicit reference. Throughout this section, all rings will be assumed Noetherian.
\subsection{Hilbert-Samuel multiplicity and Hilbert-Kunz multiplicity}

\begin{definition}
    Let $(R,\mathfrak{m})$ be a local ring of dimension $d$, $I \subset R$ an $\mathfrak{m}$-primary ideal. The {\em Hilbert-Samuel multiplicity} of $R$ with respect to $I$ is defined to be
    $$e(I;R)=\lim_{n \rightarrow +\infty}\frac{d!\cdot\lambda(R/I^n)}{n^d}, $$ where $\lambda(M)$ denotes the length of $M$ as an $R$-module. We often write $e(I)$ when the ring $R$ is understood.
\end{definition}

We have the following well-known properties:

\begin{proposition}
    Let $(R,\mathfrak{m})$ be a local ring of dimension $d$.

    (i) Let $I$ and $J$ be two $\mathfrak{m}$-primary ideals of $R$ having the same integral closures, then $e(I)=e(J)$.
    
    (ii) If $(R,\mathfrak{m})$ is Cohen-Macaulay, then $e(I)=\lambda(R/I)$ for every parameter ideal $I$. Here a parameter ideal means an $\mathfrak{m}$-primary ideal which has exactly $d$ generators.
\end{proposition}

When $R$ is of characteristic $p>0$, the Frobenius map $F: R \rightarrow R,\ x \mapsto x^p$ is a ring homomorphism. For an ideal $I \subset R$ and $e \in \mathbb{N}$, we denote $I^{[p^e]}$, the $e$-th Frobenius power of $I$, to be the ideal generated by all the $p^e$-th power of elements of $I$. Then we have a limit similar to the Hilbert-Samuel multiplicity:

\begin{definition}
    Let $(R,\mathfrak{m})$ be a local ring of dimension $d$ and characteristic $p>0$, $I \subset R$ an $\mathfrak{m}$-primary ideal. The {\em Hilbert-Kunz multiplicity} of $R$ with respect to $I$ is defined to be

    $$e_{HK}(I;R)=\lim_{e \rightarrow +\infty}\frac{\lambda(R/I^{[p^e]})}{p^{ed}}. $$ We often write $e_{HK}(I)$ when the ring $R$ is understood.
\end{definition}

We also have properties for the Hilbert-Kunz multiplicity:

\begin{proposition}\label{eHK-property}
    Let $(R,\mathfrak{m})$ be a local ring of dimension $d$ and characteristic $p>0$.

    (i) Let $I$ and $J$ be two $m$-primary ideals of $R$ having the same tight closures, then $e_{HK}(I)=e_{HK}(J)$.

    (ii) If $(R,\mathfrak{m})$ is a regular local ring, then $e_{HK}(I)=\lambda(R/I)$.
    
    (iii) $\frac{e(I)}{d!} \leq e_{HK}(I) \leq e(I)$ for all $\mathfrak{m}$-primary ideals $I$.

    (iv) If $I$ is a parameter ideal, then $e_{HK}(I)=e(I)$.
\end{proposition}

\subsection{$s$-multiplicity}

In \cite{s-multiplicity}, William D. Taylor investigates a function that interpolates continuously between Hilbert-Samuel multiplicity and Hilbert-Kunz multiplicity, which is called $s$-multiplicity. We give a brief review here.

\begin{proposition}\textup{(\cite[Theorem 2.1]{s-multiplicity})}
    Let $(R,\mathfrak{m})$ be a local ring of dimension $d$ and characteristic $p>0$. $I,J \subset R$ are $\mathfrak{m}$-primary ideals. For $s>0$, the limit
    $$h_s(I,J)=\lim_{e \rightarrow +\infty} \frac{\lambda(R/(I^{\lceil sp^e \rceil}+J^{[p^e]}))}{p^{ed}} $$ exists.
\end{proposition}
    
Here are some properties of $h_s(I,J)$:

\begin{proposition}\label{h-property}\textup{(\cite[Proposition 2.6, Theorem 2.7]{s-multiplicity})}
Let $(R,\mathfrak{m})$ be a local ring of dimension $d$ and characteristic $p>0$. Let $I,J \subset R$ be $\mathfrak{m}$-primary ideals and let $s>0$ be a real number.

(i) $h_s(I,J)\leq \min\{\frac{s^d}{d!}e(I),e_{HK}(J)\}$.

(ii) If $s'\geq s$ then $h_{s'}(I,J)\geq h_s(I,J)$.

(iii) If $I'$ and $J'$ are ideals of $R$ such that $I\subseteq I'$ and $J\subseteq J'$, then $h_s(I',J')\leq h_s(I,J)$. 

(iv) If $I'$ is an ideal of $R$ with the same integral closure as $I$, then $h_s(I',J)=h_s(I,J)$.

(v) If $J'$ is an ideal of $R$ with the same tight closure as $J$, then $h_s(I,J')=h_s(I,J)$.

(vi) The function $h_s(I,J)$ is Lipschitz continuous.
\end{proposition}

There are still some important properties related to two thresholds.

\begin{definition}
    Let $R$ be a ring of characteristic $p>0$, let $I,J$ be ideals of $R$. For $e \in \mathbb{N}$, set
    $$\nu^I_J(p^e)=\sup \{n \in \mathbf{N} \mid I^n \nsubseteq J^{[p^e]}\} \quad \mathrm{and} \quad \mu^I_J(p^e)=\inf \{n \in \mathbb{N}\mid J^{[p^e]} \nsubseteq I^n\}. $$
    We also set $$c_J(I)=\lim_{e \rightarrow +\infty} \frac{\nu^I_J(p^e)}{p^e} \quad\mathrm{and} \quad b_J(I)=\lim_{e \rightarrow +\infty} \frac{\mu^I_J(p^e)}{p^e}.$$
    Furthermore, if $(R,\mathfrak{m})$ is a local ring and $I,J \subset R$ are $\mathfrak{m}$-primary ideals, then $b_J(I) \leq c_J(I)$.
\end{definition}

\begin{proposition}\label{h-shouwei-property}\textup{(\cite[Lemma 3.3]{s-multiplicity})}
    Let $(R,\mathfrak{m})$ be a local ring of dimension $d$ and characteristic $p>0$, and let $I,J \subset R$ be $\mathfrak{m}$-primary ideals.

    (i) If $s\leq b_J(I)$ then $h_s(I,J)=\frac{s^d}{d!}e(I)$.

    (ii) If $s \geq c_J(I)$ then $h_s(I,J)=e_{HK}(J)$.
\end{proposition}

For normalization, the following normalization factor needs to be introduced.

\begin{proposition}\textup{(\cite[Proposition 3.4]{s-multiplicity})}
    Let $k$ be a field of characteristic $p>0$, $R=k[[x_1,\dots,x_d]], \ \mathfrak{m}=(x_1,\dots,x_d).$ Then
    $$h_s(\mathfrak{m},\mathfrak{m})=\sum_{i=0}^{\lfloor s \rfloor}\frac{(-1)^i}{d!}\binom{d}{i}(s-i)^d.$$ We denote $h_s(\mathfrak{m},\mathfrak{m})$ by $\mathcal{H}_s(d)$.
\end{proposition}

\begin{remark}\label{H-measure}
    $\mathcal{H}_s(d)$ is exactly the measure of the set $$\{(x_1,\dots,x_d) \in [0,1]^d \mid x_1+\dots+x_d \leq s\} \subset \mathbb{R}^d.$$
\end{remark}

Here are some properties of $\mathcal{H}_s(d)$.

\begin{proposition}\label{H-property}\textup{(\cite[Lemma 3.7]{s-multiplicity})}
    (i) $\mathcal{H}_s(d)$ is non-decreasing.

    (ii) $\mathcal{H}_s(d)$ is Lipschitz continuous.

    (iii) If $s \geq d$, $\mathcal{H}_s(d)=1$.

    (iv) If $0<s \leq 1$, $\mathcal{H}_s(d)=\frac{s^d}{d!}$.

    (v) $H_s(d)+H_{d-s}(d)=1$ for $0 < s \leq d$. 

    (vi) For $d \geq 1$, $\mathcal{H}_s(d)=\int_{s-1}^{s}\mathcal{H}_{t}(d-1)dt$.
\end{proposition}

\begin{corollary}\label{cor-H'}
    We have $\mathcal{H}'_s(d)=\mathcal{H}_{s}(d-1)-\mathcal{H}_{s-1}(d-1)$ and $\mathcal{H}'_s(d)=\mathcal{H}'_{d-s}(d)$.
\end{corollary}
\begin{proof}
    By Proposition \ref{H-property}(v) and (vi), the result follows directly.
\end{proof}

Now we can introduce the definition of $s$-multiplicity:

\begin{definition}\textup{(\cite[Definition 3.5]{s-multiplicity})}
    Let $(R,\mathfrak{m})$ be a local ring of dimension $d$ and characteristic $p>0$. $I,J \subset R$ are $\mathfrak{m}$-primary ideals. For $s>0$, the {\em $s$-multiplicity} with respect to $(I,J)$ is defined to be
    $$e_s(I,J)=\frac{h_s(I,J)}{\mathcal{H}_s(d)}.$$
\end{definition}

By Proposition \ref{h-property} and \ref{h-shouwei-property}, we have the following corollary.

\begin{corollary}\label{e-property}\textup{(\cite[Corollary 3.8, Corollary 3.9]{s-multiplicity})}
    Let $(R,\mathfrak{m})$ be a local ring of dimension $d$ and characteristic $p>0$ and let $I,J \subset R$ be $\mathfrak{m}$-primary ideals.

    (i) If $s \leq \min \{1,b_J(I)\}$ then $e_s(I,J)=e(I)$.

    (ii) If $s\geq \max\{d,c_J(I)\}$ then $e_s(I,J)=e_{HK}(J)$.

    (iii) If $I'$ is an ideal of $R$ with the same integral closure as $I$, then $e_s(I',J)=e_s(I,J)$.

    (iv) If $J'$ is an ideal of $R$ with the same tight closure as $J$, then $e_s(I,J')=e_s(I,J)$.

    (v) The function $e_s(I,J)$ is Lipschitz continuous.

    (vi) If $(R,\mathfrak{m})$ is Cohen-Macaulay and $I$ is a parameter ideal, then $e_s(I,I)=e(I)=\lambda(R/I)$ for all $s > 0$.
\end{corollary}

Next we consider $e_s(I,I)$ and denote it by $e_s(I)$.

\begin{corollary}
    Let $(R,\mathfrak{m})$ be a local ring of dimension $d$ and characteristic $p>0$ and let $I\subset R$ be an $\mathfrak{m}$-primary ideal generated by $r$ elements. Then $c_I(I) \leq r$ and for $s \geq r$, $e_s(I)=e_{HK}(I)$.
\end{corollary}
\begin{proof}
    Clearly we have $I^{rp^e+1} \subseteq I^{[p^e]}$ for every $e \in \mathbb{N}$. Therefore $c_I(I) \leq \lim_{e \rightarrow \infty} \frac{rp^e+1}{p^e}=r$. The rest of the result follows from Corollary \ref{e-property}(ii).
\end{proof}

Then we recall the definition of $s$-closure.

\begin{definition}
    Let $R$ be a ring of characteristic $p>0$. Let $I \subset R$ be an ideal and $s \geq 1$ be a real number. An element $x \in R$ is said to be in the {\em weak $s$-closure} of $I$ if there exists $c \in R^{\circ}$ such that for all $e \gg 0$, $cx^{p^e} \in I^{\lceil sp^e \rceil}+I^{[p^e]}$, where $R^{\circ}$ denotes the complement of the union of the minimal primes of $R$. We denote the set of all $x$ in the weak $s$-closure of $I$ by $I^{\text{w.cl}_s}$. Clearly we have $I \subseteq I^{\text{w.cl}_s}$.

    We define $s$-closure of $I$ as the union of the following chain of ideals:
    $$I \subseteq I^{\text{w.cl}_s} \subseteq (I^{\text{w.cl}_s})^{\text{w.cl}_s} \subseteq ((I^{\text{w.cl}_s})^{\text{w.cl}_s})^{\text{w.cl}_s} \subseteq \dots .$$ We denote this ideal by $I^{\text{cl}_s}$.
\end{definition}

With a slight modification of the proof of \textup{\cite[Theorem 4.6]{s-multiplicity}}, we can show that

\begin{proposition}\label{s-closure-criterion}
    Let $(R,\mathfrak{m})$ be a local ring of characteristic $p>0$ and let $I$ and $J$ be $\mathfrak{m}$-primary ideals of $R$ with $I \subseteq J$. If $J \subseteq I^{\mathrm{cl}_s}$, then $e_s(J)=e_s(I)$. If $R$ is an $F$-finite complete domain, then $e_s(I,J)=e_s(I)$ implies $J \subseteq I^{\mathrm{cl}_s}$ and $I^{\mathrm{cl}_s}=I^{\mathrm{w.cl}_s}$.
\end{proposition}

In \cite{lower-bound-of-s}, L. E. Miller and W. D. Taylor give a lower bound of $e_s(I)$.

\begin{theorem}\label{lower-bound}\textup{(\cite[Theorem 3.8]{lower-bound-of-s})}
    Let $(R,\mathfrak{m})$ be a Cohen-Macaulay local ring of dimension $d$ and characteristic $p>0$. Let $I \subset R$ be an $\mathfrak{m}$-primary ideal and $J$ be a reduction of $I$ that is a parameter ideal. Let $J^*$ be the tight closure of $J$. If $I/J^*$ is generated by $r_0$ elements, then we have
    $$e_s(I) \geq \frac{\mathcal{H}_t(d)-r\mathcal{H}_{t-1}(d)}{\mathcal{H}_s(d)}e(I) $$ for every $r \geq r_0$ and $1 \leq t \leq s$.
\end{theorem}

\iffalse
We can directly obtain a corollary from Theorem \ref{lower-bound}, which answers Question 2.9 in \cite{question-less-d} under the assumption of Cohen-Macaulay.
\begin{corollary}
    Let $(R,\mathfrak{m})$ be a Cohen-Macaulay local ring of positive characteristic and $I \subset R$ be an $\mathfrak{m}$-primary ideal. Then $\frac{e(I)}{e_{HK}(I)} < d!$.
\end{corollary}
\begin{proof}
    Choosing $s$ equal to the number of generators of $I$, we have $s \geq d$ and $$\frac{e(I)}{e_{HK}(I)}\leq \frac{1}{\mathcal{H}_t(d)-r\mathcal{H}_{t-1}(d)}.$$
    Set $f_r(t)=\mathcal{H}_t(d)-r\mathcal{H}_{t-1}(d)$. For $1 \leq t \leq 2$, we have
    $$f_r(t)=\frac{t^d}{d!}-\frac{(t-1)^d}{(d-1)!}-r\frac{(t-1)^d}{d!} \quad \mathrm{and} \quad f'_r(t)=\frac{t^{d-1}-(d+r)(t-1)^{d-1}}{(d-1)!}.$$ We can see $f_r(1)=\frac{1}{d!}$ and $f'_r(1)>0$. Thus, $$\frac{e(I)}{e_{HK}(I)} \leq \frac{1}{\max_{1\leq t \leq 2}f_r(t)}<\frac{1}{f_r(1)}=d!\ .$$
\end{proof}
\fi

\section{Milnor number and Tjurina number of hypersurface singularities in positive characteristic}
Throughout this section, $K$ always denotes an arbitrary field with characteristic $p>0$. We denote the formal power series ring $K[[x_1,\dots,x_n]]$, $n \geq 2$ by $R$ and the maximal ideal $(x_1,\dots,x_n)$ by $\mathfrak{m}$.

\subsection{Generalized Milnor number}

For a hypersurface singularity $f \in R$, its Jacobian ideal $j(f)$ and Tjurina ideal $tj(f)$ are defined as
$$j(f)=(\frac{\partial f}{\partial x_1},\dots, \frac{\partial f}{\partial x_n}) \quad \text{and} \quad tj(f)=(f)+j(f), $$
its Milnor number $\mu(f)$ and Tjurina number $\tau(f)$ are defined in (\ref{def-mu-tau}). We say $f$ is an {\em isolated hypersurface singularity} if $\tau(f) < \infty$.

In fields of characteristic $0$, $\tau(f)<\infty$ if and only if $\mu(f)<\infty$. It fails in characteristic $p$. Taking $f=x^p+y^{p+1} \in K[[x,y]]$, we can see that $\tau(f)=p^2$ and $\mu(f)=\infty$. However, in the sense of multiplying by a unit, $\mu$ and $\tau$ can both be finite. For example, $\mu((1+x)f)=\tau((1+x)f)=p^2$. In \cite{mu-O}, A. Hefez, J. H. O. Rodrigues, and R. Salom\~ao constructed a generalized Milnor number which is always finite for isolated hypersurface singularities in arbitrary characteristic. We collect some results here.

\begin{definition}
    For an isolated hypersurface singularity $f \in R=K[[x_1,\dots,x_n]]$ with coordinate ring $\mathcal{O}_f$, we define the generalized Milnor number as
    $$\mu(\mathcal{O}_f)=e(tj(f)),$$ where $e(tj(f))$ is the Hilbert-Samuel multiplicity of $tj(f)$.
    %$$\mu(\mathcal{O}_f)=\min\{\mu(uf) \mid u \in R^{\times}\}, $$ where $R^{\times}$ denotes the units of $R$.
\end{definition}

Here are some properties of $\mu(\mathcal{O}_f)$.

\begin{proposition}\textup{(\cite[Corollary 4.6]{mu-O})}\label{mu-O-u}
    (i) If $K$ is an infinite field, then $\mu(\mathcal{O}_f)=\min\{\mu(uf) \mid u \in R^{\times}\},$ and there exists $u \in R^{\times}$ such that $\mu(uf)=\mu(\mathcal{O}_f)$.
    
    %$\mu(\mathcal{O}_f)$ is well defined, i.e., there exists $u \in R^{\times}$ such that $\mu(uf)=\mu(\mathcal{O}_f)$.

    (ii) In characteristic $0$, $\mu(\mathcal{O}_f)=\mu(f)$ since $\mu(f)=e(j(f))=e(tj(f))$ automatically holds.

    (iii) $\mu(\mathcal{O}_f)=\mu(uf)$ if and only if $uf \in \overline{j(uf)}$. Here $\overline{I}$ means the integral closure of $I$.
\end{proposition}

\begin{theorem}\label{mu-O-all-u}
    For an isolated hypersurface singularity $f \in R$, if $f^l \in \mathfrak{m}tj(f)^l$ for some $l \geq 1$, then $f \in \overline{j(f)}$ and $\mu(\mathcal{O}_f)=\mu(uf)$ for every $u \in R^{\times}$. 
\end{theorem}

\begin{proof}
    \textup{\cite[Theorem 4.8]{mu-O}} proved the equivalence of $f^l \in \mathfrak{m}tj(f)^l$ and $\mu(\mathcal{O}_f)=\mu(uf)$ over any algebraically closed field. Here we provide a proof over arbitrary fields.

    Write $g=uf$, then $tj(g)=tj(f)$. $f^l \in \mathfrak{m}tj(f)^l$ implies $g^l \in \mathfrak{m}tj(g)^l$. Since $tj(g)=j(g)+g$, we have $tj(g)^l=j(g)tj(g)^{l-1}+g^l$. Applying Nakayama's Lemma on $tj(g)^l/j(g)tj(g)^{l-1}$, we know $tj(g)^l=j(g)tj(g)^{l-1}$. Hence, $j(g)$ is a reduction of $tj(g)$. Therefore, $\mu(uf)=\lambda(R/j(g))=e(j(g))=e(tj(g))=e(tj(f))=\mu(\mathcal{O}_f).$
\end{proof}

\begin{remark}\label{remark-all-u}
    In characteristic $0$, we always have $f \in \overline{j(f)}$, see Proposition 7.1.5 of \cite{Integral-closure}. Therefore, the statement in Theorem \ref{mu-O-all-u} holds automatically in zero characteristic.
\end{remark}

\subsection{Quotient of generalized Milnor number and Tjurina number in positive characteristic}
In this part, we can state one of our main results.

\begin{theorem}\label{main-thm-p}
    Let $f \in R=K[[x_1,\dots,x_n]],\ n \geq 2$ be an isolated hypersurface singularity. Then we have
    \begin{equation}\label{main-leq}
        \frac{\mu(\mathcal{O}_f)}{\tau(f)} \leq \frac{1}{\mathcal{H}_{\frac{n+1}{2}}(n)-\mathcal{H}_{\frac{n-1}{2}}(n)}.
    \end{equation}
\end{theorem}

\begin{proof}
    If $K$ is finite, choose $L=\overline{K}$ the algebraic closure of $K$, $S=L[[x_1,\dots,x_n]]$. Write $f_{L}=f \in S$. Then $tj(f_{L})=tj(f)S$, and $S/tj(f_L)^r \cong (R/tj(f)^r)\otimes_K L$ for $r \geq 1$. Thus, $\lambda_S(S/tj(f_L)^r)=\lambda_R(R/tj(f)^r)$. Choosing $r=1$ we obtain $\tau(f_L)=\tau(f)$. For every $r$ using the definition of Hilbert-Samuel multiplicity, we obtain $e_S(tj(f_L))=e_R(tj(f))$. That is, $\mu(\mathcal{O}_{f_L})=\mu(\mathcal{O}_f)$. Therefore, we can assume that $K$ is infinite without loss of generality. Since $\mu(\mathcal{O}_f)$ and $\tau(f)$ are invariants in the sense of multiplying a unit, we can also assume that $f \in \overline{j(f)}$ by Proposition \ref{mu-O-u}.

    By definition, $\mu(\mathcal{O}_f)=e(tj(f))$. Since $(R,\mathfrak{m})$ is a regular local ring, $\tau(f)=\lambda(R/tj(f))=e_{HK}(tj(f))$ by Proposition \ref{eHK-property}(ii). Therefore, we have $$\frac{\mu(\mathcal{O}_f)}{\tau(f)} =\frac{e(tj(f))}{e_{HK}(tj(f))}.$$

    Since $f \in \overline{j(f)}$, $j(f)$ is a reduction of $tj(f)$, and $tj(f)/j(f)$ is generated by at most one element. Therefore, applying Theorem \ref{lower-bound} for $s=n+1,r=1$ (or by \cite[Theorem 3.1]{extra-e_HK}), we have
    \begin{equation}\label{main-choose-t}
        e_{HK}(tj(f))=e_{n+1}(tj(f)) \geq (\mathcal{H}_t(n)-\mathcal{H}_{t-1}(n))e(tj(f)).
    \end{equation}

    Next we consider $f_n(s)=\mathcal{H}_s(n)-\mathcal{H}_{s-1}(n),1 \leq s \leq n+1$. We use induction on $n$ to show $f_n(s)$ has only one local maximum. $n=1$ follows easily. For $n \geq 2$ and suppose that $u  \in (0,n)$ is the unique local maximum such that $f_{n-1}(s)$ is strictly increasing on $(0,u)$ and is strictly decreasing on $(u,n)$. Let $t \in (0,n+1)$ be a local maximum of $f_n(s)$. Therefore, $0=f_n'(t)=f_{n-1}(t)-f_{n-1}(t-1)$ by Proposition \ref{H-property}(vi). Since $f_{n-1}(s)$ is strictly increasing on $(0,u)$ and is strictly decreasing on $(u,n)$, there is a unique $t$ such that $f_{n-1}(t)=f_{n-1}(t-1)$. Therefore, the local maximum $t$ of $f_n(s)$ is unique.

    Choosing $t$ to be the unique local maximum of $f_n(t)$, we have
    $$f_n(t)=\mathcal{H}_t(n)-\mathcal{H}_{t-1}(n)=(1-\mathcal{H}_{n-t}(n))-(1-\mathcal{H}_{n+1-t}(n))=\mathcal{H}_{n+1-t}(n)-\mathcal{H}_{n-t}(n)=f_n(n+1-t). $$ Thus, $n+1-t=t$ implies $t=\frac{n+1}{2}$.

    Therefore, choosing $t=\frac{n+1}{2}$ in (\ref{main-choose-t}), we get
    $$\frac{\mu(\mathcal{O}_f)}{\tau(f)} =\frac{e(tj(f))}{e_{HK}(tj(f))} \leq \frac{1}{\mathcal{H}_{\frac{n+1}{2}}(n)-\mathcal{H}_{\frac{n-1}{2}}(n)}.$$
\end{proof}

\begin{corollary}
    For a plane curve singularity $f \in K[[x,y]]$, putting $n=2$ in (\ref{main-leq}), we get $\frac{\mu(\mathcal{O}_f)}{\tau(f)} \leq \frac{4}{3}$. For a surface singularity $f \in K[[x,y,z]]$, putting $n=3$ we get $\frac{\mu(\mathcal{O}_f)}{\tau(f)} \leq \frac{3}{2}$. For $n=4$ and $n=5$, the ratios $\frac{\mu(\mathcal{O}_f)}{\tau(f)} \leq \frac{192}{115}$ and $ \frac{20}{11}$, respectively.
\end{corollary}

We have the following estimate for this bound:
\begin{proposition}\label{bound-less-n}
    For $n \geq 2$, $\frac{\mu(\mathcal{O}_f)}{\tau(f)} \leq \frac{1}{\mathcal{H}_{\frac{n+1}{2}}(n)-\mathcal{H}_{\frac{n-1}{2}}(n)}<n$.
\end{proposition}
\begin{proof}
    We have already seen in the proof of Theorem \ref{main-thm-p} that $f_n(t)=\mathcal{H}_t(n)-\mathcal{H}_{t-1}(n)$ has the unique maximum $f_n(\frac{n+1}{2})$. Since $\mathcal{H}_0(n)=0$ and $\mathcal{H}_n(n)=1$, we have
    $$1=\sum_{i=1}^nf_n(i) < n f_n(\frac{n+1}{2}).$$
    Therefore, $$ \frac{\mu(\mathcal{O}_f)}{\tau(f)} \leq \frac{1}{\mathcal{H}_{\frac{n+1}{2}}(n)-\mathcal{H}_{\frac{n-1}{2}}(n)}=\frac{1}{f_n(\frac{n+1}{2})}<n.$$
\end{proof}

\begin{remark}\label{=4/3-example}
    Dealing with the equality in (\ref{main-leq}) is more complicated. In fact, let $f=x^{2p+1}+y^{2p+1}+x^{p}y^{p} \in K[[x,y]]$ with $\mathrm{char}K=p>0$. Then $j(f)=(x^{2p},y^{2p}),\ tj(f)=(x^{2p},y^{2p},x^py^p)$. In this case, $(x^py^p)^2 \in j(f)$ implies that $j(f)$ is the minimal reduction of $tj(f)$, hence $\mu(\mathcal{O}_f)=\mu(f)=4p^2,\ \tau(f)=3p^2$. Then $$\frac{\mu(\mathcal{O}_f)}{\tau(f)}=\frac{e(tj(f))}{e_{HK}(tj(f))}=\frac{4}{3},$$ which shows that the equality in (\ref{main-leq}) can be achieved.
\end{remark}

However, we have the following criterion, Proposition \ref{lemma-eq-hold}. First we give two lemmas.

\begin{lemma}\label{q=aN+b}
    Let $(R, \mathfrak{m})$ be an arbitrary Noetherian local ring of dimension $n$ and characteristic $p$. Let $I$ be an $\mathfrak{m}$-primary ideal, $f \in I$. If $f^N \in I^{N+1}$ for some $N \geq 1$, then $f^q \in I^{q+\lfloor \frac{q}{N} \rfloor}$ for every $q=p^e$. 
\end{lemma}

\begin{proof}
    Write $q=aN+b,\ 0 \leq b \leq N-1$. Then $f^q=f^{aN+b}=(f^N)^af^b \in I^{(N+1)a+b}= I^{q+\lfloor \frac{q}{N} \rfloor}$.
\end{proof}

\begin{lemma}\label{lemma-extension-s-closure}
    For any field extension $L/K$, write $S=L[[x_1,\dots,x_n]]$. Then for $A \subset R$ ideals, we have $(AS)^{\mathrm{w.cl}_s} \cap R=A^{\mathrm{w.cl}_s}$.
\end{lemma}
\begin{proof}
    Clearly $A^{\mathrm{w.cl}_s} \subseteq (AS)^{\mathrm{w.cl}_s} \cap R$. We need to show $(AS)^{\mathrm{w.cl}_s} \cap R \subseteq A^{\mathrm{w.cl}_s}$. Write $a \in R$ belongs to the left side, there exists nonzero $c \in S$ such that
    \begin{equation}\label{extension-s-closure}
        ca^q \in A^{\lceil sq \rceil}S+A^{[q]}S
    \end{equation}
    for all sufficiently large $q=p^e$. 
    
    Write $c=\sum_{\alpha}c_{\alpha}\mathbf{x}^{\alpha} \in S=L[[x_1,\ldots,x_n]]$ and choose $c_{\alpha_0} \neq 0$. Extend $c_{\alpha_0}$ to a $K$-basis of $L$ and choose $\Lambda:L \rightarrow K$ to be the $K$-linear map that sends $c_{\alpha_0}$ to $1$ and every other basis element to $0$. Then $\Lambda$ induces an $R$-linear map $\Lambda: S \rightarrow R$. Applying $\Lambda$ to (\ref{extension-s-closure}), we get 
    $$\Lambda(c)a^q \in A^{\lceil sq \rceil}+A^{[q]}$$ for all sufficiently large $q=p^e$ and $\Lambda(c) \neq 0$. Thus, $a\in A^{\mathrm{w.cl}_s}$.
\end{proof}

\begin{proposition}\label{lemma-eq-hold}
    Let $f \in R=K[[x_1,\dots,x_n]]$ be an isolated hypersurface singularity and $f \in \overline{j(f)}$. Moreover, assume that one of the following conditions holds:

    (i) $n \geq 3,$ and $j(f):f \nsubseteq j(f)^{\mathrm{cl}_{\frac{n-1}{2}}}$;

    (ii) there exists $l \geq 1$ such that $f^l \in \mathfrak{m}tj(f)^{l}$;\\
    then the equality in (\ref{main-leq}) cannot be attained. That is, $\frac{\mu(\mathcal{O}_f)}{\tau(f)} < \frac{1}{\mathcal{H}_{\frac{n+1}{2}}(n)-\mathcal{H}_{\frac{n-1}{2}}(n)}$.
\end{proposition}

\begin{proof}
    If $f \in j(f)$, then the equality cannot hold since $\mu(f)=\tau(f)$. Next we assume that $f \notin j(f)$, thus $j(f):f$ is an $\mathfrak{m}$-primary ideal of $R$.
    
     By Lemma \ref{lemma-extension-s-closure} and the discussion in Theorem \ref{main-thm-p}, we can assume that $K$ is algebraically closed. Moreover, we can assume that $f \in \overline{j(f)}$. We denote $I=tj(f),\ J=j(f)$. We denote $q=p^e$ and by $q \rightarrow \infty$, we mean $q$ moves through $e \rightarrow \infty$. 

    We have
    \begin{equation}\label{total-three-terms}
        \begin{aligned}
            \lambda(\frac{I^{\lceil sq \rceil}+I^{[q]}}{J^{\lceil sq \rceil}+J^{[q]}}) &\leq \lambda(\frac{I^{\lceil tq \rceil}+I^{[q]}}{J^{\lceil sq \rceil}+J^{[q]}})\\
            &=\lambda(\frac{I^{\lceil tq \rceil}+I^{[q]}}{I^{\lceil tq \rceil}+J^{[q]}})+\lambda(\frac{I^{\lceil tq \rceil}+J^{[q]}}{J^{\lceil tq \rceil}+J^{[q]}})+\lambda(\frac{J^{\lceil tq \rceil}+J^{[q]}}{J^{\lceil sq \rceil}+J^{[q]}}).
        \end{aligned}
    \end{equation}
    For the second term, since $J$ is a reduction of $I$, $\lambda(\frac{I^{\lceil tq \rceil}+J^{[q]}}{J^{\lceil tq \rceil}+J^{[q]}})=o(q^n)$. For the third term, since $J$ is a parameter ideal and $(R,\mathfrak{m})$ is Cohen-Macaulay, $$\lim_{q \rightarrow \infty} \frac{1}{q^n}\lambda(\frac{J^{\lceil tq \rceil}+J^{[q]}}{J^{\lceil sq \rceil}+J^{[q]}})=h_s(J)-h_t(J)=(\mathcal{H}_s(n)-\mathcal{H}_t(n))e(J)=(\mathcal{H}_s(n)-\mathcal{H}_t(n))e(I).$$

    Finally, for the first term, if $f$ satisfies (i), we have

    \begin{equation}
        \begin{aligned}
            \lambda(\frac{I^{\lceil tq \rceil}+I^{[q]}}{I^{\lceil tq \rceil}+J^{[q]}}) &=\lambda(\frac{I^{\lceil tq \rceil}+J^{[q]}+f^q}{I^{\lceil tq \rceil}+J^{[q]}})\\
            &=\lambda(\frac{R}{({I^{\lceil tq \rceil}+J^{[q]}}):f^q})\\
            &\leq \lambda(\frac{R}{(I^{\lceil tq \rceil}:f^q)+J^{[q]}:f^q})\\
            &\leq \lambda(\frac{R}{I^{\lceil (t-1)q \rceil}+(J:f)^{[q]}}).
        \end{aligned}
    \end{equation}
    Therefore, we have
    \begin{equation}\label{first-term-leq}
    \begin{aligned}
        \lim_{q \rightarrow \infty} \frac{1}{q^n}\lambda(\frac{I^{\lceil tq \rceil}+I^{[q]}}{I^{\lceil tq \rceil}+J^{[q]}})& \leq \lim_{q \rightarrow \infty} \frac{1}{q^n}\lambda(\frac{R}{I^{\lceil (t-1)q \rceil}+(J:f)^{[q]}})\\
        &=h_{t-1}(I,(J:f))=\mathcal{H}_{t-1}(n)e_{t-1}(I,(J:f))=\mathcal{H}_{t-1}(n)e_{t-1}(J,(J:f))\\
        &\leq \mathcal{H}_{t-1}(n)e_{t-1}(J,J)=\mathcal{H}_{t-1}(n)e(J)=\mathcal{H}_{t-1}(n)e(I).
    \end{aligned}
    \end{equation}
    Following (\ref{total-three-terms}), all the above together implies
    $$h_s(J)-h_s(I)=\lim_{q \rightarrow \infty}\frac{1}{q^n}\lambda(\frac{I^{\lceil sq \rceil}+I^{[q]}}{J^{\lceil sq \rceil}+J^{[q]}}) \leq \mathcal{H}_{t-1}(n)e(I)+(\mathcal{H}_s(n)-\mathcal{H}_t(n))e(I).$$
    Thus,
    \begin{equation}\label{total-sum}
        \begin{aligned}
            e_s(I)=\frac{h_s(I)}{\mathcal{H}_s(n)} &\geq \frac{1}{\mathcal{H}_s(n)}(h_s(J)-(\mathcal{H}_{t-1}(n)+\mathcal{H}_s(n)-\mathcal{H}_t(n))e(I))\\
            &=(\frac{\mathcal{H}_t(n)-\mathcal{H}_{t-1}(n)}{\mathcal{H}_{s}(n)})e(I).
        \end{aligned}
    \end{equation}

    If the equality in (\ref{main-leq}) holds, then the equality in (\ref{total-sum}) holds for $s=n+1,t=\frac{n+1}{2}$ as shown in the proof of Theorem \ref{main-thm-p}, which in turn implies that the equality in (\ref{first-term-leq}) holds. That is, $e_{\frac{n-1}{2}}(J,(J:f))=e_{\frac{n-1}{2}}(J,J)$. Since $K$ is algebraically closed, $(R, \mathfrak{m})$ is an $F$-finite complete domain. Therefore, by Proposition \ref{s-closure-criterion}, we have $(J:f) \subseteq J^{\text{cl}_{\frac{n-1}{2}}}$, which contradicts (i).

    Next we consider the case where $f$ satisfies (ii). Since $I$ is $\mathfrak{m}$-primary, we choose $C$ such that $\mathfrak{m}^C \subseteq I$. Then we have $f^{Cl} \in \mathfrak{m}^CI^{Cl} \subseteq I^{Cl+1}$. Thus, for $N \geq Cl+1,$ we have $f^N \in I^{N+1}$.
    
    For the first term in (\ref{total-three-terms}), we write $u_q=\lceil tq \rceil-q-\lfloor \frac{q}{N} \rfloor$. We can choose $N$ large enough such that $u_q \geq 0$. By Lemma \ref{q=aN+b}, $I^{u_q} \cdot f^q \in I^{\lceil tq \rceil}$. Similarly as above, we have
    \begin{equation}
        \begin{aligned}
            \lambda(\frac{I^{\lceil tq \rceil}+I^{[q]}}{I^{\lceil tq \rceil}+J^{[q]}}) &=\lambda(\frac{I^{\lceil tq \rceil}+J^{[q]}+f^q}{I^{\lceil tq \rceil}+J^{[q]}})\\
            &=\lambda(\frac{R}{({I^{\lceil tq \rceil}+J^{[q]}}):f^q})\\
            &\leq \lambda(\frac{R}{(I^{\lceil tq \rceil}:f^q)+J^{[q]}:f^q})\\
            &\leq \lambda(\frac{R}{I^{u_q}+J^{[q]}}).
        \end{aligned}
    \end{equation}
    
    Since $\lim_{q \rightarrow \infty} \frac{u_q}{q}=t-1-\frac{1}{N}$, we have
    \begin{equation}
    \begin{aligned}
        \lim_{q \rightarrow \infty} \frac{1}{q^n}\lambda(\frac{I^{\lceil tq \rceil}+I^{[q]}}{I^{\lceil tq \rceil}+J^{[q]}})&\leq \lim_{q \rightarrow \infty} \frac{1}{q^n}\lambda(\frac{R}{I^{u_q}+J^{[q]}})\\
        &=h_{t-1-\frac{1}{N}}(I,J)=\mathcal{H}_{t-1-\frac{1}{N}}(n)e_{t-1-\frac{1}{N}}(I,J)=\mathcal{H}_{t-1-\frac{1}{N}}(n)e(J)\\
        &=\mathcal{H}_{t-1-\frac{1}{N}}(n)e(I).
    \end{aligned}
    \end{equation}
    Then
    $$h_s(J)-h_s(I)=\lim_{q \rightarrow \infty}\frac{1}{q^n}\lambda(\frac{I^{\lceil sq \rceil}+I^{[q]}}{J^{\lceil sq \rceil}+J^{[q]}}) \leq \mathcal{H}_{t-1-\frac{1}{N}}(n)e(I)+(\mathcal{H}_s(n)-\mathcal{H}_t(n))e(I).$$
    Thus,
    \begin{equation}
        \begin{aligned}
            e_s(I)=\frac{h_s(I)}{\mathcal{H}_s(n)} &\geq \frac{1}{\mathcal{H}_s(n)}(h_s(J)-(\mathcal{H}_{t-1-\frac{1}{N}}(n)+\mathcal{H}_s(n)-\mathcal{H}_t(n))e(I))\\
            &=(\frac{\mathcal{H}_t(n)-\mathcal{H}_{t-1-\frac{1}{N}}(n)}{\mathcal{H}_{s}(n)})e(I)\\
            &>(\frac{\mathcal{H}_t(n)-\mathcal{H}_{t-1}(n)}{\mathcal{H}_{s}(n)})e(I).
        \end{aligned}
    \end{equation}
    Taking $s=n+1$ and $t=\frac{n+1}{2}$, we get $$\frac{\mu(\mathcal{O}_f)}{\tau(f)} =\frac{e(tj(f))}{e_{HK}(tj(f))} < \frac{1}{\mathcal{H}_{\frac{n+1}{2}}(n)-\mathcal{H}_{\frac{n-1}{2}}(n)},$$ which shows the equality in (\ref{main-leq}) cannot be attained.
\end{proof}

Thus, we propose the following conjecture.
\begin{conjecture}\label{conj-j-cl-closure}
    Let $f \in K[[x_1,\dots,x_n]]$ define an isolated hypersurface singularity, $n \geq 3$. Let $j(f)$ denote its Jacobian ideal. Assume that $f\in \overline{j(f)}$. Then $j(f):f \nsubseteq j(f)^{\mathrm{cl}_{\frac{n-1}{2}}}$.
\end{conjecture}
If Conjecture \ref{conj-j-cl-closure} holds for $n \geq 3$, then $$\frac{\mu(\mathcal{O}_f)}{\tau(f)} < \frac{1}{\mathcal{H}_{\frac{n+1}{2}}(n)-\mathcal{H}_{\frac{n-1}{2}}(n)}.$$

Since $j(f)^{\mathrm{cl}_{\frac{n-1}{2}}}$ is hard to calculate, one may consider $\overline{j(f)}$ instead of $j(f)^{\mathrm{cl}_{\frac{n-1}{2}}}$ because we have $j(f)^{\mathrm{cl}_{\frac{n-1}{2}}} \subseteq \overline{j(f)}$ by definition. We can easily prove the case $n=2$:

\begin{proposition}\label{j(f)-bar-n=2}
    Let $f \in K[[x_1,x_2]]$ define an isolated hypersurface singularity over a field of positive characteristic and let $j(f)$ denote its Jacobian ideal. Assume that $f \in \overline{j(f)}$. Then $j(f):f \nsubseteq \overline{j(f)}$.
\end{proposition}
\begin{proof}
    If not, then $(j(f):f) \subseteq \overline{j(f)}$. We have $\overline{j(f)}^2 \subseteq j(f)$ by Skoda's Theorem. Since $f \in \overline{j(f)}$, we have $\overline{j(f)} \subseteq (j(f):f)$. Thus, we get $(j(f):f)=\overline{j(f)}$.

    We have the exact sequence
    $$0 \rightarrow \frac{j(f):f}{j(f)} \rightarrow \frac{R}{j(f)} \xrightarrow{\cdot f} \frac{R}{j(f)} \rightarrow \frac{R}{tj(f)} \rightarrow 0. $$
    Then we have $\lambda( R/(j(f):f))=\lambda (R/j(f))-\lambda(R/tj(f))=\mu(\mathcal{O}_f)-\tau(f)$.

    Therefore, we have
    $$2 \geq \frac{e(j(f):f)}{e_{HK}(j(f):f)}=\frac{e(\overline{j(f)})}{\lambda(R/(j(f):f))}=\frac{\mu(\mathcal{O}_f)}{\mu(\mathcal{O}_f)-\tau(f)}=\frac{1}{1-\frac{\tau(f)}{\mu(\mathcal{O}_f)}} \geq 4, $$
    where the first and last inequalities come from Proposition \ref{eHK-property}(iii) and Theorem \ref{main-thm-p}. Thus, it leads to a contradiction.
\end{proof}

\begin{remark}
    Proposition \ref{j(f)-bar-n=2} does not contradict Proposition \ref{lemma-eq-hold} and the example in Remark \ref{=4/3-example}, since Proposition \ref{lemma-eq-hold} also needs the condition $n \geq 3$ to show the equality cannot be attained.
\end{remark}

However, there is a counterexample in \cite{conjecture-example} showing that $j(f):f \nsubseteq \overline{j(f)}$ does not always hold when $n=5$. Let $f=\frac{1}{4}\sum_{i=1}^5 x_i^4+(\sum_{i=1}^5x_i)^5 \in K[[x_1,\dots,x_5]]$, where $\mathrm{char}K =11$. One can show $\overline{j(f)}=\mathfrak{m}^3$ and $j(f):f \subseteq \overline{j(f)}$.

\section{The quotient of Milnor number and Tjurina number in characteristic zero}\label{section-char0}
In this section, we obtain our main result in characteristic zero from Theorem \ref{main-thm-p} by reduction modulo $p$. Moreover, we can show that the upper bound of $\frac{\mu}{\tau}$ cannot be attained in characteristic zero. In the following, $K$ denotes a field of characteristic zero, $R=K[[x_1,\dots,x_n]]$ and $\mathfrak{m}=(x_1,\dots,x_n)$.

First we recall the finite determinacy theorem.

\begin{definition}
    Two power series $f,g \in R$ are called {\em right equivalent} (denoted by $f \sim_r g$) if there exists $\phi \in Aut(R)$ such that $f=\phi(g)$.

    Two power series $f,g \in R$ are called {\em contact equivalent} (denoted by $f \sim_c g$) if there exist $\phi \in Aut(R)$ and $U \in R^{\times}$ such that $f=U\cdot\phi(g)$.
\end{definition}
Note that when $f$ is isolated, both $\mu(f)$ and $\tau(f)$ are invariants under right (or contact) equivalence in characteristic zero.

A power series $f \in R$ is called {\em right (or contact) $k$-determined} if for all $g \in R$ such that $f-g \in \mathfrak{m}^{k+1}$, we have $f \sim_r g$ (or $f \sim_c g$). $f$ is called finitely determined if there exists a $k$ such that $f$ is $k$-determined.

\begin{theorem}\label{finite-deter}\textup{(\cite[Theorem 4.6]{finite-deter2})} 
    For $f \in \mathfrak{m}$, $f$ is finitely contact determined if and only if $\tau(f)<\infty$.
\end{theorem}
    
Therefore, any isolated hypersurface singularity is finitely contact determined.

Next, we recall the semicontinuity of the Milnor number and Tjurina number.

\begin{proposition}\label{semi-cont}\textup{(\cite[Proposition 63]{semi-tau})}
    Let $A$ be Noetherian and $F \in A[\mathbf{x}]$. For $\mathfrak{p} \in \mathrm{Spec}A$,  denote by $k(\mathfrak{p})=A_{\mathfrak{p}}/\mathfrak{p}A_{\mathfrak{p}}$, $F(\mathfrak{p})$ the image of $F$ in $k(\mathfrak{p})[[\mathbf{x}]]$ and $\mu(F(\mathfrak{p}))$ (resp. $\tau(F(\mathfrak{p}))$) the Milnor number (resp. Tjurina number) of $F(\mathfrak{p})$ in $k(\mathfrak{p})$. Then $\mu(F(\mathfrak{p}))$ and $\tau(F(\mathfrak{p}))$ are semicontinuous at $\mathfrak{p} \in \mathrm{Spec}A$. That is, $\mathfrak{p}$ has an open neighborhood $U \subset \mathrm{Spec}A$ such that $\mu(F(\mathfrak{p})) \geq \mu(F(\mathfrak{q}))$ and $\tau(F(\mathfrak{p})) \geq \tau(F(\mathfrak{q}))$ for all $\mathfrak{q} \in U$.
\end{proposition}

From Proposition \ref{semi-cont}, we directly obtain the following corollary.

\begin{corollary}\label{cor-semicont}
    Let $A$ be a Noetherian domain and $F \in A[\mathbf{x}]$. Denote $\eta$ the generic point of $\mathrm{Spec}A$. If $\mu(F(\eta))$ and $\tau(F(\eta))$ are finite, then there exists an open dense subset $U \subset \mathrm{Spec}A$ such that $$\mu(F(\eta))=\mu(F(\mathfrak{p})),\ \tau(F(\eta))=\tau(F(\mathfrak{p}))$$ for all $\mathfrak{p} \in U$.
\end{corollary}

Then we present the main result of this section.

\begin{theorem}\label{main-thm-0}
    Let $f \in R=K[[x_1,\dots,x_n]],\ n \geq 2$ be an isolated hypersurface singularity. Then we have
    \begin{equation}\label{eq-main-char0}
        \frac{\mu(f)}{\tau(f)} < \frac{1}{\mathcal{H}_{\frac{n+1}{2}}(n)-\mathcal{H}_{\frac{n-1}{2}}(n)}.
    \end{equation}
\end{theorem}

\begin{proof}
    Denote $S=K[x_1,\dots,x_n]_{(x_1,\dots,x_n)}$ and write $\mathfrak{m}_S$ the maximal ideal of $S$. Then $\hat{S}=K[[x_1,\dots,x_n]]$. By Theorem \ref{finite-deter}, we can assume $f \in K[x_1,\dots,x_n]$ without changing $\mu(f),\tau(f)$. Write $J=j(f)\subset \hat{S},\ I=tj(f) \subset \hat{S},$ and $I_S=(f,\frac{\partial f}{\partial x_1},\dots, \frac{\partial f}{\partial x_n}) \subset S.$ Then $I=I_S \hat{S}$.

    By Theorem \ref{mu-O-all-u} and Remark \ref{remark-all-u}, there exists $l \geq 1$ such that $f^l \in \mathfrak{m}tj(f)^l=\mathfrak{m}I^l=(\mathfrak{m}_SI_S^{l})\hat{S}$. Since $S \rightarrow \hat{S}$ is faithfully flat, $f^l\in \mathfrak{m}_SI_S^l$ in $S$. %Thus we can assume that $f^N \in I_S^{N+1}$ for some $N$ since $I_S$ is $\mathfrak{m}_S$-primary.

    Write $$f^l=\sum_{|\bm{\alpha}|=l,\ 1\leq i \leq n}\frac{a_{\bm{\alpha},i}[\mathbf{x}]}{b_{\bm{\alpha},i}[\mathbf{x}]}x_ih^{\bm{\alpha}},$$ where ${a_{\bm{\alpha},i}[\mathbf{x}]},{b_{\bm{\alpha},i}[\mathbf{x}]} \in K[x_1,\dots,x_n]$ and $h^{\bm{\alpha}}$ is a product of $|\bm{\alpha}|$ elements chosen from $f,\frac{\partial f}{\partial x_1},\dots, \frac{\partial f}{\partial x_n}$.

    Let $A$ be the finitely generated $\mathbb{Z}$-algebra obtained by adjoining all the coefficients of $f, a_{\bm{\alpha},i},b_{\bm{\alpha},i}$ to $\mathbb{Z}$, then $f \in A[x_1,\dots,x_n]$. Let $F=\text{Frac}A$, we have $F \subset K$. Set $S_F=F[x_1,\dots,x_n]_{(x_1,\dots,x_n)}$. Then $S_F \rightarrow S$ is faithfully flat. Denote $\eta$ the generic point of $\text{Spec}A$. It follows that $\mu_F(f(\eta))=\mu_K(f)<\infty$ and $\tau_F(f(\eta))=\tau_K(f)<\infty$, where for $g \in k[x_1,\dots,x_n]$, $\mu_k(g)$ (resp. $\tau_k(g)$) means the Milnor number (resp. Tjurina number) of $g \in k[x_1,\dots,x_n]$.
    
    Moreover, $f^l=\sum_{|\bm{\alpha}|=l,\ 1\leq i \leq n}\frac{a_{\bm{\alpha},i}[\mathbf{x}]}{b_{\bm{\alpha},i}[\mathbf{x}]}x_ih^{\bm{\alpha}}$ in $A[x_1,\dots,x_n]_{(\mathbf{x})}$. By clearing denominators, we obtain 
    \begin{equation}\label{Df}
        D(\mathbf{x})f^l=\sum_{|\bm{\alpha}|=l,\ 1\leq i \leq n}a'_{\bm{\alpha},i}[\mathbf{x}]x_ih^{\bm{\alpha}} \in A[[x_1,\dots,x_n]],
    \end{equation}
    where $D(\mathbf{x})=\Pi {b_{\bm{\alpha},i}[\mathbf{x}]}$ and $D(0) \neq 0$.

    Consider all closed points $\mathfrak{p} \in \mathrm{Spec}A$. There exists a prime number $p \in \mathbb{Z}$ such that $\mathfrak{p} \cap\mathbb{Z}=(p)$. Therefore, $k(\mathfrak{p})=A_{\mathfrak{p}}/\mathfrak{p}A_{\mathfrak{p}}$ is a field of characteristic $p>0$. All closed points form a dense subset of $\mathrm{Spec}A$. For $a \in A,$ we denote by $\overline{a}$ the image of $a$ in $k(\mathfrak{p})$.

    Therefore, we can take closed point $\mathfrak{p} \in \mathrm{Spec}A$ simultaneously satisfying the following conditions:\\
    (i) $\overline{D(0)} \neq 0$,\\
    (ii) $\mu(f(\eta))=\mu(f(\mathfrak{p}))$,\\
    (iii) $\tau(f(\eta))=\tau(f(\mathfrak{p}))$,\\
    because the $\mathfrak{p}$ satisfying (i), (ii), (iii) respectively form a non-empty open subset by Corollary \ref{cor-semicont}. 
    %Here $k(\mathfrak{p})$ is not required to be infinite by Remark \ref{k-infinite-p}.

    Write $$R_A=A[[x_1,\dots,x_n]],\  \mathfrak{m}_A=(x_1,\dots,x_n) \subset R_A$$ and 
    $$J_A=(\frac{\partial f}{\partial x_1},\dots, \frac{\partial f}{\partial x_n}) \subset R_A,\ I_A=(f,J_A) \subset R_A.$$

    Modulo $\mathfrak{p}$ in (\ref{Df}) and denote by $\mathfrak{a}(\mathfrak{p})$ the image of ideal $\mathfrak{a} \subset R_A$ in $k(\mathfrak{p})[[\mathbf{x}]]$, we get $D(\mathfrak{p})(\mathbf{x})f(\mathfrak{p})^l \in \mathfrak{m}_A(\mathfrak{p}) I_A(\mathfrak{p})^{l}$. Since $D(\mathfrak{p})(0)=\overline{D(0)} \neq 0$ in $k(\mathfrak{p})$, $D(\mathfrak{p})(\mathbf{x})$ is invertible in $k(\mathfrak{p})[[x]]$, hence we have $f(\mathfrak{p})^l \in \mathfrak{m}_A(\mathfrak{p})I_A(\mathfrak{p})^{l}$. By Theorem \ref{mu-O-all-u}, $f(\mathfrak{p}) \in \overline{J_A(\mathfrak{p})}$ and $\mu(f(\mathfrak{p}))=\mu(\mathcal{O}_{f(\mathfrak{p})})$. Therefore, by applying Proposition \ref{lemma-eq-hold} in $k(\mathfrak{p})[[x]]$, we have $$\frac{\mu(f)}{\tau(f)}=\frac{\mu(f(\eta))}{\tau(f(\eta))}=\frac{\mu(\mathcal{O}_{f(\mathfrak{p})})}{\tau(f(\mathfrak{p}))} < \frac{1}{\mathcal{H}_{\frac{n+1}{2}}(n)-\mathcal{H}_{\frac{n-1}{2}}(n)}.$$
\end{proof}

\begin{remark}
    For $n \geq 2$, we have $\frac{\mu(f)}{\tau(f)} < \frac{1}{\mathcal{H}_{\frac{n+1}{2}}(n)-\mathcal{H}_{\frac{n-1}{2}}(n)}<n$ by Proposition \ref{bound-less-n}, which is better than the bound given by Y. Liu \cite{liuyongqiang}.
\end{remark}

Putting $n=2$, we get $\frac{\mu(f)}{\tau(f)} < \frac{4}{3}$ for an isolated plane curve singularity $f$ over a characteristic zero field. Applying the same method as Proposition \ref{j(f)-bar-n=2}, we can show
\begin{corollary}
    Let $f \in K[[x_1,x_2]]$ define an isolated hypersurface singularity and let $j(f)$ denote its Jacobian ideal. Assume that $f \in \overline{j(f)}$. Then $j(f):f \nsubseteq \overline{j(f)}$.
\end{corollary}

Putting $n=3$ and $K=\mathbb{C}$, we have answered Conjecture \ref{conj-3/2}:
\begin{corollary}\label{cor-C-3/2}
    $\frac{\mu}{\tau}<\frac{3}{2}$ for any isolated hypersurface singularity $f \in \mathbb{C}[[x,y,z]]$.
\end{corollary}

The following Wahl's theorem gives the relationship among $\mu$, $\tau$, and the geometric genus $p_g$.

\begin{theorem}\label{Wahl-thm}\textup{(\cite[Corollary 2.9]{Wahl})}
    Let $X$ be an isolated complete intersection singularity of $\dim 2$ over $\mathbb{C}$. Then 
    $$\mu-\tau \leq 2p_g-\dim H^1(A;\mathbb{C}),$$ where $A$ is the exceptional divisor of a minimal good resolution of $X$.
\end{theorem}

Then we can see that Corollary \ref{cor-C-3/2} is in fact a weak version of Durfee's conjecture.

\begin{proposition}\label{durfee-3/2}\textup{(\cite[Proposition 4.11]{quotient-4/3})}
    Durfee's Conjecture \ref{conj-durfee} implies Corollary \ref{cor-C-3/2}.
\end{proposition}
\begin{proof}
    If Conjecture \ref{conj-durfee} is true, then $6p_g < \mu$. By Theorem \ref{Wahl-thm}, $\mu-\tau \leq 2p_g < \frac{\mu}{3}$. That is, $\frac{\mu}{\tau} < \frac{3}{2}$.
\end{proof}
Therefore, the validity of Corollary \ref{cor-C-3/2} makes Conjecture \ref{conj-durfee} more likely to hold. 

\section{The minimal Tjurina number of the equisingularity class of $x_1^n+\dots+x_d^n$}

In \cite{greuel-conj}, A. Dimca and G.M. Greuel constructed a sequence of hypersurface singularities in $\mathbb{C}[[x,y]]$ whose initial forms are $x^n+y^n$ and for which $\frac{\mu}{\tau}$ tends to $\frac{4}{3}$. In \cite{Wahl}, Wahl showed that there exists a sequence in $\mathbb{C}[[x,y,z]]$ with initial forms $x^n+y^n+z^n$ such that $\frac{\mu}{\tau} \rightarrow \frac{3}{2}$ when $n \rightarrow \infty$. In this part, we will calculate $\frac{\mu}{\tau}$ of hypersurface singularities in $K[[x_1,\dots,x_d]]$ over any field of characteristic zero with initial forms $x_1^n+\dots+x_d^n$ following the process in \cite{Zariski-moduli} p.114-127. We omit lengthy calculations and proofs, and only present the main results.

In the following, we denote by $K$ a field with $\mathrm{char}K=0$, $\mathfrak{m}=(x_1,\dots, x_d)$ the maximal ideal of $K[[x_1,\dots,x_d]]$. We denote by $\#A$ the number of elements of the set $A$. We say $a\sim_k b$ for $k>0$ if $\frac{a-b}{n^k} \rightarrow 0$ when $n \rightarrow \infty$. 

First we give two lemmas.

\begin{lemma}\label{lemma-number}
    For $n \in \mathbb{N}$, $\#\{a_1+\dots+a_d=n \mid \forall a_i \in \mathbb{N}\}=\binom{n+d-1}{d-1}$. In particular, there are $\binom{n+d-1}{d-1}$ monomials of degree $n$ in $\mathfrak{m}^n$. Moreover, 
    \begin{equation}
        \begin{aligned}
            \#\{a_1+\dots+a_d=n \mid \forall a_i \in \mathbb{N}\} &\sim_{d-1} \frac{n^{d-1}}{(d-1)!},\\
            \#\{a_1+\dots+a_d \leq n \mid \forall a_i \in \mathbb{N}\} &\sim_{d} \frac{n^{d}}{d!}.
        \end{aligned}
    \end{equation}
\end{lemma}
\begin{proof}
    The proof is elementary.
\end{proof}

\begin{lemma}\label{lemma-number<n}
    For $\lambda>0$, $b_n,m_n$ are sequences of non-negative integers such that $\frac{b_n}{n} \rightarrow 1$ and $\frac{m_n}{n} \rightarrow \lambda$ as $n \rightarrow \infty$. Then we have
    \begin{equation}\label{set-H'}
        \begin{aligned}
            \#\{a_1+\dots+a_d \leq m_n \mid \forall a_i \in \mathbb{N},a_i \leq b_n\} &\sim_{d} n^d \cdot \mathcal{H}_\lambda(d),\\
            \#\{a_1+\dots+a_d = m_n \mid \forall a_i \in \mathbb{N},a_i \leq b_n\} &\sim_{d-1} n^{d-1} \cdot \mathcal{H}'_\lambda(d).
        \end{aligned}
    \end{equation}
\end{lemma}
\begin{proof}
    We will use the following fact:
    $$\frac{\#\bigl(Q\cap n^{-1}\mathbb{Z}^r\bigr)}{n^r}
\longrightarrow \operatorname{vol}_r(Q) $$ for every bounded polytope $Q\subseteq\mathbb{R}^r$.

Put $c_n=b_n/n$, $t_n=m_n/n$ and set $$P_n=\{\boldsymbol{x}\in[0,c_n]^d
\;|\;x_1+\cdots+x_d\leq t_n \}. $$ Then 
$$\#\{a_1+\dots+a_d \leq m_n \mid \forall a_i \in \mathbb{N},a_i \leq b_n\}
= \#\bigl(P_n\cap n^{-1}\mathbb{Z}^d\bigr). $$
Since $P_n$ converges to
\[P_\lambda=\{\boldsymbol{x}\in[0,1]^d\;|\;
x_1+\cdots+x_d\leq\lambda\},\] we obtain 
\[\frac{1}{n^d}\#\{a_1+\dots+a_d \leq m_n \mid \forall a_i \in \mathbb{N},a_i \leq b_n\}
\longrightarrow
\operatorname{vol}_d(P_\lambda)
=
\mathcal{H}_\lambda(d).
\]

For the second assertion, the equation $$a_1+\cdots+a_d=m_n$$
uniquely determines $a_d$. Hence the required number equals
\[
\#\bigl(R_n\cap n^{-1}\mathbb{Z}^{d-1}\bigr),
\]
where
\[
R_n=\{\boldsymbol{x}\in[0,c_n]^{d-1}\;|\;t_n-c_n\leq x_1+\cdots+x_{d-1}\leq t_n\}.
\]
The polytopes $R_n$ converge to
$$R_\lambda=\{\boldsymbol{x}\in[0,1]^{d-1}\;|\;\lambda-1\leq x_1+\cdots+x_{d-1}\leq\lambda\}.$$
Therefore,
\[
\frac{1}{n^{d-1}}
\#\{a_1+\dots+a_d = m_n \mid \forall a_i \in \mathbb{N},a_i \leq b_n\}
\longrightarrow
\operatorname{vol}_{d-1}(R_\lambda).
\]
It remains to calculate the volume.

For a real number $a$, denote $(a)_+=\max\{a,0\}$. For a set $A \subset \mathbb{R}$, denote $\mathbf{1}_A$ the indicator function of $A$. Since $$\mathcal{H}_t(d)=\int_{[0,1]^{d-1}}\min\left\{1,\left(t-\sum_{i=1}^{d-1}x_i\right)_+\right\}\,dx_1\dots dx_{d-1},$$ we have
$$\mathcal{H}'_t(d)=
\int_{[0,1]^{d-1}}
\mathbf{1}_{\left\{
t-1<\sum_{i=1}^{d-1}x_i<t
\right\}}
\,dx_1\dots dx_{d-1}=
\operatorname{vol}_{d-1}(R_t).$$
Consequently,
$$\frac{1}{n^{d-1}}
\#\{a_1+\dots+a_d = m_n \mid \forall a_i \in \mathbb{N},a_i \leq b_n\}
\longrightarrow \mathcal{H}'_\lambda(d). $$
\end{proof}

\iffalse
\begin{lemma}\label{lemma-number<n}
    For $\lambda>0$, we have
    \begin{equation}\label{set-H'}
        \begin{aligned}
            \#\{a_1+\dots+a_d \leq \lambda n \mid \forall a_i \in \mathbb{N},a_i \leq n\} &\sim_{d} n^d \cdot \mathcal{H}_\lambda(d),\\
            \#\{a_1+\dots+a_d = \lfloor \lambda n \rfloor \mid \forall a_i \in \mathbb{N},a_i \leq n\} &\sim_{d-1} n^{d-1} \cdot \mathcal{H}'_\lambda(d).
        \end{aligned}
    \end{equation}
\end{lemma}
\begin{proof}
    For every $(a_1,\dots,a_d) \in \{a_1+\dots+a_d \leq \lambda n \mid \forall a_i \in \mathbb{N},a_i \leq n\},$ $(\frac{a_1}{n},\dots,\frac{a_d}{n}) \in X_s=\{(x_1,\dots,x_d) \in [0,1]^d \mid x_1+\dots+x_d \leq s\}$, whose measure is equal to $\mathcal{H}_\lambda(d)$ by Remark \ref{H-measure}. When $n \rightarrow \infty$, the image of $(\frac{a_1}{n},\dots,\frac{a_d}{n})$ is dense in $X_s$. Hence $\#\{a_1+\dots+a_d \leq \lambda n \mid \forall a_i \in \mathbb{N},a_i \leq n\} \sim_{d} n^d \cdot \mathcal{H}_\lambda(d)$. The second formula in (\ref{set-H'}) follows by differentiating with respect to $\lambda$. 
\end{proof}
\fi

Then we can present the main result of this section.

\begin{theorem}
    Let $K$ be a field of characteristic $0$. Denote by $\tau_{min}$ the minimal Tjurina number of positive weight deformations of $x_1^n+\dots+x_d^n$. Then $$\tau_{min}\sim_d (\mathcal{H}_{\frac{d+1}{2}}(d)-\mathcal{H}_{\frac{d-1}{2}}(d))\cdot n^d.$$ Therefore, for every $d\geq 2$, there exists a sequence in $K[[x_1,\dots,x_d]]$ consisting of deformations of $x_1^n+\dots+x_d^n$ such that $$\frac{\mu}{\tau} \rightarrow \frac{1}{\mathcal{H}_{\frac{d+1}{2}}(d)-\mathcal{H}_{\frac{d-1}{2}}(d)} \mathrm{\ \ when\ n \rightarrow \infty}.$$ As a result, the bounds in Theorem \ref{main-thm-0} and Theorem \ref{main-thm-p} are sharp.
\end{theorem}

\begin{proof}
    We omit the lengthy calculation and proof, and only present the main results. One can refer to \cite{Zariski-moduli} p.114-127 for details.
~\\
~\\
\textbf{I.} Write $f_0=x_1^n+\dots+x_d^n$. A basis of the Milnor algebra $K[[x_1,\dots,x_d]]/j(f_0)$ is given by $x_1^{a_1}\dots x_d^{a_d},\ 0 \leq a_i \leq n-2$. Write
\begin{equation}
    \begin{aligned}
        f_u&=x_1^n+\dots+x_d^n+\sum_{\substack{0 \leq a_i \leq n-2,\\ a_1+\dots+a_d \geq n+1}} u_{(a_1,\dots,a_d)}x_1^{a_1}\dots x_d^{a_d}\\
        &=x_1^n+\dots+x_d^n+f_{n+1}+f_{n+2}+\dots,
    \end{aligned}
\end{equation}
where $$f_{n+p+1}=\sum_{\substack{0 \leq a_i \leq n-2,\\ a_1+\dots+a_d= n+p+1}} u_{(a_1,\dots,a_d)}x_1^{a_1}\dots x_d^{a_d},\ p \geq 0.$$
Then the minimal Tjurina number among positive weight deformations of $f_0$ equals $$\dim_K K[[x_1,\dots,x_d]]/tj(f_u)$$ for general $u=(u_{(a_1,\dots,a_d)})$. We denote $T_u=tj(f_u)$. 
\\
~\\
\textbf{II.} Choose $g_{n+p+1}$ a homogeneous polynomial of degree $n+p+1$ and write
$$g_{n+p+1}=P_p^*f_{n+1}+\sum_{i=1}^d Q_{i,p+2}^* x_i^{n-1}, $$ where $P_p^*,Q_{i,p+2}^*$ are homogeneous polynomials of degree $p,p+2$, respectively. 
Write $$P_p^*=\sum_{a_1+\dots+a_d=p}P_{(a_1,\dots,a_d)}^*x_1^{a_1}\dots x_d^{a_d}, \quad Q_{i,p+2}^*=\sum_{a_1+\dots+a_d=p+2}Q_{i,(a_1,\dots,a_d)}^*x_1^{a_1}\dots x_d^{a_d}.$$ Note that if there is a monomial $x_1^{a_1}\dots x_d^{a_d}$ in $P_p^*$ such that $a_i \geq n-1$, we merge this monomial with $Q_{i,p+2}^*$. Thus, every monomial $x_1^{a_1}\dots x_d^{a_d}$ in $P_p^*$ satisfies $a_i \leq n-2$ for every $i$.

Denote $r_{n,p}=\#\{a_1+\dots+a_d=p\mid \forall a_i\in \mathbb{N},a_i \leq n-2\}$. All coefficients $P_{(a_1,\dots,a_d)}^*, Q_{i,(a_1,\dots,a_d)}^* \in K$, and there are $r_{n,p}$ different $P_{(a_1,\dots,a_d)}^*$, $\binom{p+1+d}{d-1}$ different $Q_{i,(a_1,\dots,a_d)}^*$ for each $i$.

For $\iota=(a_1,\dots,a_d)$ such that $a_1+\dots+a_d=n+p+1$ and $a_i \leq n-2$, denote $x^{(p)}_{\iota}=x_1^{a_1}\dots x_d^{a_d}$. Note that there are $r_{n,n+p+1}$ such $\iota$. Denote $l^{(p)}_{\iota}$ the coefficient of $x^{(p)}_\iota$ in $g_{n+p+1}$. Denote
\begin{equation}
    \begin{aligned}
        l^{(p)}_{\iota}&=\sum l^{(p)}_{\iota,(a_1,\dots,a_d)}(u)P^*_{(a_1,\dots,a_d)}
    \end{aligned}
\end{equation}
Denote $L_p(u)$ the $r_{n,n+p+1} \times r_{n,p}$ matrix given by $\bigg(l^{(p)}_{\iota,(a_1,\dots,a_d)}(u)\bigg)$. We can show that $L_p(u)$ are of full rank for general $u$. That is,
$$\mathrm{rank}L_p(u)=\min\{r_{n,p},r_{n,n+p+1}\}.$$
%$$\mathrm{rank}L^*_p(u)=\#\{a_1+\dots+a_d=p\mid \forall a_i\in \mathbb{N},a_i \leq n-2\}+d \cdot \binom{p+1+d}{d-1}. $$
\\
~\\
\textbf{III.} Suppose $u$ is general (such that $L_p(u)$ is of full rank). We denote by $T$ the ideal $T_u$. One can show that every homogeneous polynomial $g_{n+p+1} \in \mathfrak{m}^{n+p+2}+T$ of degree $n+p+1$ has the form 
\begin{equation}\label{g-P-Q}
    g_{n+p+1}=P_p^*f_{n+1}+\sum_{i=1}^d Q_{i,p+2}^* x_i^{n-1}.
\end{equation}
Conversely, every homogeneous polynomial $g_{n+p+1}$ of degree $n+p+1$ with the form (\ref{g-P-Q}) belongs to $\mathfrak{m}^{n+p+2}+T$. Therefore, $$B_p=\{x_1^{a_1}\dots x_d^{a_d}f_{n+1} \mid a_1+\dots+a_d=p,\ a_i \leq n-2\}$$ along with monomials $$A_p=\{x_1^{a_1}\dots x_d^{a_d}x_i^{n-1} \mid a_1+\dots+a_d=p+2\} $$ generate the vector space of homogeneous polynomials $g_{n+p+1}$ that belong to $\mathfrak{m}^{n+p+2}+T$. There remain $\binom{n+p+d}{d-1}-\#A_p$ monomials of degree $n+p+1$, while each element in $B_p$ induces a relation between these monomials modulo $\mathfrak{m}^{n+p+2}+T$. Therefore, the dimension of $\mathfrak{m}^{n+p+1}/\mathfrak{m}^{n+p+1} \cap (\mathfrak{m}^{n+p+2}+T)$ is $\binom{n+p+d}{d-1}-\#A_p-\mathrm{rank}L_p(u)=\max {\{r_{n,n+p+1}-r_{n,p},0\}}$.
\\
~\\
\textbf{IV.} Let $n \rightarrow \infty$. Then $\mathfrak{m}^{n+p+1} \subset T$ if $r_{n,n+p+1}\leq r_{n,p}$, that is
\begin{equation}\label{eq-m-in-J}
    \#\{a_1+\dots+a_d=p\mid \forall a_i\in \mathbb{N},a_i \leq n-2\} \geq \#\{a_1+\dots+a_d=n+p+1\mid \forall a_i\in \mathbb{N},a_i \leq n-2\}
\end{equation}
We find when the equality in (\ref{eq-m-in-J}) holds. Let $p=\lambda n$. By Lemma \ref{lemma-number} and \ref{lemma-number<n}, the equality in (\ref{eq-m-in-J}) infers that 
\begin{equation}\label{eq-m-in-J-lambda}
    \mathcal{H}'_{1+\lambda}(d)=\mathcal{H}'_{\lambda}(d).
\end{equation}
We have $\mathcal{H}'_s(d)=\mathcal{H}'_{d-s}(d)$ by Corollary \ref{cor-H'} and $\mathcal{H}'_s(d)$ has the only local maximum by the proof of Theorem \ref{main-thm-p}. Hence, $\lambda+\lambda+1=d$ and $\lambda=\frac{d-1}{2}$.

Let $r$ be the smallest integer such that $\mathfrak{m}^{n+r} \subset T$. We have $\mathfrak{m}^{n+p+1} \subset T$ for $p  \geq r-1$, $r \sim_1 \frac{d-1}{2}$ and

\begin{equation}
    \begin{aligned}
        \dim_K K[[x_1,\dots,x_d]]/T &\sim_d \dim_K K[[x_1,\dots,x_d]]/\mathfrak{m}^{n+1}+\sum_{0 \leq p \leq r-1} \dim_K \frac{\mathfrak{m}^{n+p+1}+T}{\mathfrak{m}^{n+p+2}+T}\\
        & \sim_d \frac{n^d}{d!}+\sum_{0 \leq p \leq \frac{d-1}{2} n} \bigg(\frac{(n+p)^{d-1}}{(d-1)!}-\#A_p-r_{n,p}\bigg)\\
        &\sim_d \frac{n^d}{d!}+n^d(\mathcal{H}_{\frac{d+1}{2}}(d)-\mathcal{H}_1(d))-\mathcal{H}_{\frac{d-1}{2}}(d) \cdot n^d\\
        &\sim_d (\mathcal{H}_{\frac{d+1}{2}}(d)-\mathcal{H}_{\frac{d-1}{2}}(d))\cdot n^d.
    \end{aligned}
\end{equation}

That is, $$\tau_{min}\sim_d (\mathcal{H}_{\frac{d+1}{2}}(d)-\mathcal{H}_{\frac{d-1}{2}}(d))\cdot n^d.$$
Since $\mu(f_u) \sim_d n^d$, we have $$\frac{\mu}{\tau_{min}} \rightarrow \frac{1}{\mathcal{H}_{\frac{d+1}{2}}(d)-\mathcal{H}_{\frac{d-1}{2}}(d)}.$$

Thus, the bound in Theorem \ref{main-thm-0} is sharp. As we can obtain a field of positive characteristic by reduction modulo $p$ for some $p$, the bound in Theorem \ref{main-thm-p} is also sharp for positive characteristic.
\end{proof}

Thus, we have completely solved an open problem of P. Almir\'{o}n \cite[Problem 1]{quotient-4/3} for isolated hypersurface singularities as follows:

\begin{corollary}
    Let $(X,0)\subset (\mathbb{C}^d,0)\ (d \geq 2)$ be an isolated complete intersection singularity of codimension $1$. Then $\frac{1}{\mathcal{H}_{\frac{d+1}{2}}(d)-\mathcal{H}_{\frac{d-1}{2}}(d)}$ is the optimal rational number such that $$\frac{\mu}{\tau} < \frac{1}{\mathcal{H}_{\frac{d+1}{2}}(d)-\mathcal{H}_{\frac{d-1}{2}}(d)},$$ where optimal means that there exists a family of singularities such that $\frac{\mu}{\tau}$ tends to $\frac{1}{\mathcal{H}_{\frac{d+1}{2}}(d)-\mathcal{H}_{\frac{d-1}{2}}(d)}$ when the multiplicity at the origin tends to infinity.
\end{corollary}

\newcommand{\etalchar}[1]{$^{#1}$}

\end{document}